\documentclass[reqno, 12pt]{article} 

\pdfoutput=1

\usepackage{enumerate}
\usepackage{latexsym}
\usepackage[centertags]{amsmath}
\usepackage{amsfonts}
\usepackage{amssymb}
\usepackage{amsthm}
\usepackage{newlfont}
\usepackage{graphics}
\usepackage{color}
\usepackage{float}
\usepackage{diagbox}
\usepackage[linesnumbered,ruled,vlined]{algorithm2e}
\usepackage{tikz}
\usetikzlibrary{graphs,positioning,quotes,shapes.geometric}
\usepackage{hyperref}
\usepackage{longtable}
\usepackage{rotating}
\usepackage{multirow}
\usepackage{extarrows}
\usepackage[sort,compress,numbers]{natbib}
\usepackage[utf8]{inputenc}

\newtheorem{thm}{Theorem}[section]
\newtheorem{cor}[thm]{Corollary}
\newtheorem{lem}[thm]{Lemma}
\newtheorem{prop}[thm]{Proposition}

\theoremstyle{mydefinition}

\theoremstyle{myremark}

\newtheorem{exa}[thm]{Example}
\newtheorem{prob}[thm]{Open problem}
\allowdisplaybreaks[4]

\def\CT{\mathop{\mathrm{CT}}}

\SetKwInput{KwInput}{Input}                
\SetKwInput{KwOutput}{Output}              

\newcommand\ci[1]{\mathrm{\textcircled{#1}}}

\title{A Fast Algorithm for Denumerants with Three Variables}

\author{Feihu Liu$^{1}$ and Guoce Xin$^{2,}$ \thanks{This work is partially supported by the National Natural Science Foundation of China (No.12071311).}
\\[2mm]
{\small $^{1, 2}$ School of Mathematical Sciences,}\\[-0.8ex]
{\small Capital Normal University, Beijing, 100048, P.R.~China}\\
{\small $^1$ Email address: liufeihu7476@163.com}\\
{\small $^2$ Email address: guoce\_xin@163.com}
}

\date{June 27, 2024}

\begin{document}

\maketitle

\begin{abstract}
Let $a,b,c$ be distinct positive integers such that $a<b<c$ and $\gcd(a,b,c)=1$. For any non-negative integer $n$, the denumerant function $d(n;a,b,c)$ denotes the number of solutions of the equation $ax_1+bx_2+cx_3=n$ in non-negative integers $x_1,x_2,x_3$. 
We present an algorithm that computes $d(n;a,b,c)$ with a time complexity of $O(\log b)$.
\end{abstract}

\noindent
\begin{small}
 \emph{Mathematic subject classification}: Primary 05A17; Secondary 05A15, 11P81.
\end{small}

\noindent
\begin{small}
\end{small}

\section{Introduction}
Let $A=\{a_1,a_2,...,a_k\}$ be a nonempty set of positive integers. The \emph{restricted partition function of the set} $A$, denoted $p_A(n)$, counts the number of solutions in non-negative integers $x_i$, $1\leq i\leq k$, to the equation
$$a_1x_1+a_2x_2\cdots +a_kx_k=n.$$
When $\gcd(A) = 1$, the function $p_A(n)$ is referred to as the \emph{denumerant}, denoted as $d(n;a_1,\ldots,a_k)$. This concept was introduced by Sylvester \cite{J. J. Sylvester}. A related problem is to determine the largest non-negative integer $F(A)$ for which $d(F(A); a_1, \ldots, a_k) = 0$. This integer is known as the \emph{Frobenius number} of $A$.

For the case $k=2$, Sylvester \cite{J. J. Sylvester1} derived the formula for the Frobenius number $F(A)=a_1a_2-a_1-a_2$. Sert\"oz \cite{Sertoz} and Tripathi \cite{A.Tripathi} independently obtained an explicit formula for $d(n;a_1,a_2)$. However, Curtis \cite{F.Curtis} demonstrated that for $k \geq 3$, the Frobenius number cannot be expressed by closed formulas of a certain type.

For $k=3$, let $a_1<a_2<a_3$. Greenberg \cite{Greenberg} proposed a fast algorithm for computing $F(A)$ with a runtime complexity of $O(\log a_1)$.
In 1953, Popoviciu \cite{Popoviciu} developed an algorithm with a time complexity of $O(a_3\log a_3)$ for calculating $d(n; a_1, a_2, a_3)$ when $a_1, a_2, a_3$ are pairwise coprime positive integers. 
In 1995, Lison\"ek \cite{Lisonek} presented an arithmetic procedure to compute $d(n;a_1,a_2,a_3)$ in time complexity of $O(a_1a_2\log a_2)$ when $a_1,a_2,a_3$ are pairwise coprime positive integers (under certain conditions, this complexity can be reduced to $O(a_1a_2)$).
In 2003, Brown, Chou, and Shiue \cite{Brown} introduced an $O(a_1a_2\log a_3)$ algorithm for computing $d(n;a_1,a_2,a_3)$.
In 2018, Aguil\'o-Gost and Llena \cite{AguiloGost} proposed an $O(a_2+\log a_3)$ algorithm to compute $d(n;a_1,a_2,a_3)$.
In 2020, Binner \cite{Binner} developed an algorithm for computing $d(n;a_1,a_2,a_3)$ with a runtime of $O(\log a_3)$. Some relevant references can be found in \cite{GostGarcia10,Komatsu03,Nathanson}.

In this paper, we combine the constant term method mentioned in \cite{Xin15} with a key transformation technique in \cite{XinZhang} to devise a fast algorithm for calculating $d(n;a_1,a_2,a_3)$. The complexity of our algorithm is $O(\log a_2)$ with $a_1<a_2<a_3$.

The structure of this paper is as follows. 
In Section 2, we provide a brief introduction to the constant term method. 
Section 3 outlines the necessary reduction operations for our algorithm. 
Section 4 presents our algorithm and a specific example. 
Throughout the paper, $\mathbb{C}$, $\mathbb{Z}$, $\mathbb{N}$, and $\mathbb{P}$ denote the set of all complex numbers, all integers, all non-negative integers, and all positive integers, respectively.

\section{A Brief Introduction to the Constant Term Method}

In this section, we introduce two basic tools in the study of constant term theory.
To maintain brevity, we only work over the field $\mathbb{C}((\lambda))$ of Laurent series in $\lambda$.
The constant term operator $\mathrm{CT}_{\lambda}$ acts by
$$\mathop{\mathrm{CT}}\limits_{\lambda}\sum_{i=M}^{\infty}c_{i}\lambda^{i}=c_0,$$
where $M\in \mathbb{Z}$ and the coefficients $c_i\in \mathbb{C}$.
Subsequently, the denumerant is expressed as
\begin{align*}
d(n;a_1,a_2,...,a_k)=\sum_{x_i\geq 0}\mathop{\mathrm{CT}}\limits_{\lambda}\lambda^{x_1a_1+x_2a_2+\cdots +x_ka_k-n}
=\mathop{\mathrm{CT}}\limits_{\lambda}\frac{\lambda^{-n}}{(1-\lambda^{a_1})(1-\lambda^{a_2})\cdots (1-\lambda^{a_k})}.
\end{align*}
For the theory for constant terms in several variables, we refer the reader to \cite{Xin15}.

The basic problem in constant term theory is how to efficiently extract the constant term of
\begin{align}\label{e-Elambda}
E(\lambda)&=\frac{L(\lambda)}{\prod_{i=1}^k(1-z_i\lambda^{a_i})},
\end{align}
where $L(\lambda)$ is a Laurent polynomial, the variables $z_i$ are independent of $\lambda$, and $a_i\in \mathbb{P}$ for all $i$.
In the general setting, we need to clarify the series expansion of $(1-z_i \lambda^{a_i})^{-1}$ in a field of iterated Laurent series.
The situation in $\mathbb{C}((\lambda))$ is simple. For a given positive integer $k$ and complex variable $z$, we have
\begin{align*}
\frac{1}{1-z\lambda^k}= \sum_{j\geq 0} z^j\lambda^{jk};\ \ \
\frac{1}{1-z\lambda^{-k}}= \frac{-z^{-1}\lambda^{k}}{1-z^{-1}\lambda^{k}}=-\sum_{j\geq 0}z^{-(j+1)}\lambda^{(j+1)k}.
\end{align*}

Our investigation relies on the well-known partial fraction decompositions.
\begin{prop}[\cite{Xin15}]\label{propEPFD}
Let $E(\lambda)$ be defined as in \eqref{e-Elambda}, where $z_i$ are complex parameters and the denominator factors are
coprime to each other. Then $E(\lambda)$ has a partial fraction decomposition given by
\begin{equation}
E(\lambda)=P(\lambda)+\dfrac{p(\lambda)}{\lambda^m}+\sum_{i=1}^k\dfrac{A_i(\lambda)}{1-z_i\lambda^{a_i}}, \label{EE}
\end{equation}
where $P(\lambda), p(\lambda)$, and the $A_i(\lambda)$'s are all polynomials, $\deg p(\lambda)<m$, and $\deg A_{i}(\lambda)<a_i$ for all $i$.

Furthermore, the polynomial $A_i(\lambda)$ is uniquely characterized by
\begin{equation}\label{AS}
             A_i(\lambda)\equiv E(\lambda)\cdot (1-z_i\lambda^{a_i})\ \mod\langle 1-z_i\lambda^{a_i}\rangle,\ \
             \deg_{\lambda}A_i(\lambda)<a_i,
\end{equation}
where $\langle 1-z_i\lambda^{a_i}\rangle$ denotes the ideal generated by $1-z_i\lambda^{a_i}$.
\end{prop}

The following corollary is straightforward but important.
\begin{cor}\label{e-CTECoro}
Follow the notation in Proposition \ref{propEPFD}. We obtain
\begin{equation}\label{e-CTE}
  \CT_\lambda E(\lambda) = P(0) + A_1(0)+A_2(0)+\cdots+A_k(0).
\end{equation}
Moreover, if $E(\lambda)$ is a proper fraction in $\lambda$, that is, the degree of the numerator is less than the degree in the denominator,
then $P(\lambda)=0$; if $\lim_{\lambda=0} E(\lambda)$ exists, then $p(\lambda)=0$. If both conditions hold, then
$$E(0)=A_1(0)+A_2(0)+\cdots+A_k(0).$$
\end{cor}
It is important to note that in the multivariate setting, certain terms $A_i(0)$ may be omitted from \eqref{e-CTE}. 
For the basic building block $A_s(0)$, we introduce the notation
$$\mathop{\mathrm{CT}}\limits_{\lambda}\underline{\frac{1}{1-z_s\lambda^{a_s}}}E(\lambda)\cdot(1-z_s\lambda^{a_s})=A_s(0).$$
One can think that only the single underlined factor of the denominator contributes when taking the constant term in $\lambda$.
Define the \emph{type} of $\CT_\lambda E(\lambda)$ as $(a_1,\dots, a_k)$ and the \emph{type} of $\CT_\lambda E(\lambda)$ with $1-z_s\lambda^{a_s} $ underlined
as $(\ci s; a_1,\dots, a_k)$. The \emph{index} of the latter constant term is defined as $a_s$.

We require an additional tool, as stated in the following theorem.
\begin{thm}[\cite{XinZhang},The key transformation]\label{keyTR}
Let $\gcd(u,a)=1$.
For any rational function $F(\lambda)$ that is meaningful when modulo $1-z\lambda^a$, we have
$$\mathop{\mathrm{CT}}\limits_{\lambda}\underline{\frac{1}{1-z\lambda^a}}F(\lambda)
=\mathop{\mathrm{CT}}\limits_{\lambda}\underline{\frac{1}{1-z^{1/u}\lambda^a}}F(\lambda^u),$$
where $u$ is referred to as a \emph{multiplier}.
\end{thm}

For denumerant $d(n;a_1,a_2,a_3)$, we need to introduce slack variables $z_i$ to apply Proposition \ref{propEPFD}. We have 
\begin{align*}
&d(n;a_1,a_2,a_3)
\\=&\lim_{z_i\rightarrow 1}\left\{ \mathop{\mathrm{CT}}\limits_{\lambda}\frac{\lambda^{-n}}{\underline{(1-\lambda^{a_1}z_1)}(1-\lambda^{a_2}z_2)(1-\lambda^{a_3}z_3)}
+\mathop{\mathrm{CT}}\limits_{\lambda}\frac{\lambda^{-n}}{(1-\lambda^{a_1}z_1)\underline{(1-\lambda^{a_2}z_2)}(1-\lambda^{a_3}z_3)}\right. 
\\&\left. +\mathop{\mathrm{CT}}\limits_{\lambda}\frac{\lambda^{-n}}{(1-\lambda^{a_1}z_1)(1-\lambda^{a_2}z_2)\underline{(1-\lambda^{a_3}z_3)}}\right\}
\end{align*}
or (simply written as)
\begin{align}\label{DenumConbTT}
d(n;a_1,a_2,a_3)&=\lim_{z_i\rightarrow 1}\mathop{\mathrm{CT}}\limits_{\lambda}\frac{\lambda^{-n}}{\underline{(1-\lambda^{a_1}z_1)(1-\lambda^{a_2}z_2)(1-\lambda^{a_3}z_3)}}.
\end{align}

The flow chart of our algorithm is as follows.
\begin{small}
\begin{align}\label{FlowChart}
&(\ci 1;a_1,a_2,a_3)\nonumber
\\&\quad\quad\quad\downarrow \nonumber
\\&(\ci 1;a_{11},1,a_{31})\longrightarrow (\ci 1;a_{12},1,a_{32})\longrightarrow \cdots \longrightarrow (\ci 1;a_{1(p-1)},1,a_{3(p-1)})\longrightarrow {\boxed{(\textcircled{1};a_{1p},1,a_{3p})}}
\\&\quad\quad\quad\downarrow \quad\quad\quad\quad\quad\quad\quad\quad \downarrow \quad\quad\quad\quad\quad\quad\quad\quad\quad\quad\quad\quad\quad \downarrow\nonumber
\\ &{\boxed{(\textcircled{2};a_{11},1,a_{31})}}\ \ \ \ \ {\boxed{(\textcircled{2};a_{12},1,a_{32})}}\ \ \ \ \ \cdots \ \ \ \ \ {\boxed{(\textcircled{2};a_{1(p-1)},1,a_{3(p-1)})}} \nonumber
\end{align}
\end{small}

Our algorithm is based on the following observations of constant term of type $(\ci 1;a_1,a_2,a_3)$.
If $\gcd(a_1,a_2)=1$, then a constant term of type $(\ci 1;a_{1},a_{2},a_{3})$ can be transformed into
a constant term of type $(\ci 1; a_{11},1,a_{31})$ by using Theorem \ref{keyTR}, where $a_{11}=a_1$.

A constant term of type $(\ci 1; a_{1i},1,a_{3i})$, for $1\leq i\leq p-1$,  can be expressed as the sum of two distinct constant terms:
One is of type $(\ci 2; a_{1i},1,a_{3i})$ and hence simplifies to a rational function;
The other is of type $(\ci 3; a_{1i}^{\prime},1,a_{3i}^{\prime})$ and can be written as a constant term of
type $(\ci 1; a_{1(i+1)},1,a_{3(i+1)})$ with index $a_{1(i+1)}=a_{3i}^{\prime} \leq a_{1i}/2$.

In the flow chart \eqref{FlowChart}, $a_{1p}=0$ or $1$ with $p\leq \log a_1$. When $a_{1p}=0$, the constant term of type $(\ci 1;a_{1p},1,a_{3p})$ is $0$ (see Subsection \ref{RecurOper}). When $a_{1p}=1$, the constant term of type $(\ci 1;a_{1p},1,a_{3p})$ is a single rational function.

\section{Several Reductions}

Let $a, b, c$ be positive integers such that $a < b < c$ and $\gcd(a, b, c) = 1$. We derive several reductions needed in our algorithms.

\subsection{Reduction to the Case $\gcd(a,b)=1$}
Though not used in our algorithm, we can reduce $(a,b,c)$ to the pairwise coprime case by the following lemma.

\begin{lem}\label{abprime}
Let $F(\lambda)$ be a formal power series in $\lambda$ with coefficients in the field $\mathbb{C}$.
Let $n\in \mathbb{N}$ and $g,c\in \mathbb{P}$. If $\gcd (g,c)=1$, then
$$\mathop{\mathrm{CT}}\limits_{\lambda}\left(F(\lambda^{g})\cdot\frac{\lambda^{-n}}{(1-\lambda^{c})}\right)
=\mathop{\mathrm{CT}}\limits_{\lambda}\left(F(\lambda)\cdot\frac{\lambda^{\frac{-n+tc}{g}}}{(1-\lambda^{c})}\right),$$
where $c^{\prime}c\equiv 1\pmod g$, $t\equiv nc^{\prime} \pmod g$, and $0\leq t<g$.
\end{lem}
\begin{proof}
We have
\begin{align*}
\mathop{\mathrm{CT}}\limits_{\lambda}\frac{F(\lambda^{g})\cdot\lambda^{-n}}{(1-\lambda^{c})}
=\mathop{\mathrm{CT}}\limits_{\lambda}\frac{F(\lambda^{g})\cdot\lambda^{-n}\cdot(1+\lambda^{c}+\cdots +\lambda^{(g-1)c})}{(1-\lambda^{cg})}.
\end{align*}
By $\gcd(c,g)=1$ and $-n+tc\equiv 0\pmod g$, we deduce that $t\equiv c^{-1}n\pmod g$, leading to
\begin{align*}
\mathop{\mathrm{CT}}\limits_{\lambda}\frac{F(\lambda^{g})\cdot\lambda^{-n}\cdot(1+\lambda^{c}+\cdots +\lambda^{(g-1)c})}{(1-\lambda^{cg})}
=\mathop{\mathrm{CT}}\limits_{\lambda}\frac{F(\lambda^{g})\cdot\lambda^{-n}\cdot \lambda^{tc}}{(1-\lambda^{cg})}
=\mathop{\mathrm{CT}}\limits_{\lambda}\frac{F(\lambda)\cdot\lambda^{\frac{-n+tc}{g}}}{(1-\lambda^{c})}.
\end{align*}
This concludes the proof.
\end{proof}

Suppose
\begin{align*}
&g_1=\gcd(b,c),\ \ \ a_1a\equiv 1\mod g_1,\ \ \ k\equiv na_1 \mod g_1,\ \ 0\leq k<g_1,
\\&g_2=\gcd(a,c),\ \ \ b_2b\equiv 1\mod g_2,\ \ \ j\equiv nb_2 \mod g_2,\ \ 0\leq j<g_2,
\\&g_3=\gcd(a,b),\ \ \ c_3c\equiv 1\mod g_3,\ \ \ i\equiv nc_3 \mod g_3,\ \ 0\leq i<g_3.
\end{align*}
By applying Lemma \ref{abprime} iteratively, we obtain 
\begin{align}
d(n;a,b,c)&=\mathop{\mathrm{CT}}\limits_{\lambda}\frac{\lambda^{-n}}{(1-\lambda^{a})(1-\lambda^{b})(1-\lambda^{c})}
=d\left(\frac{n-ic}{g_3};\frac{a}{g_3},\frac{b}{g_3},c\right)\label{PrimeToAB}
\\&=d\left(\frac{n-ic-jb}{g_2g_3};\frac{a}{g_2g_3},\frac{b}{g_3},\frac{c}{g_2}\right)
=d\left(\frac{n-ic-jb-ka}{g_1g_2g_3};\frac{a}{g_2g_3},\frac{b}{g_1g_3},\frac{c}{g_1g_2}\right).\label{PireWsiePrime}
\end{align}

\subsection{Reduction to Two Contribution Term}\label{ReduceTwoTerm}

For \eqref{DenumConbTT}, we consider the contribution of $(1-\lambda^{c}z_3)$ to the constant term. There exists $s\in\mathbb{P}$ such that $0<sc-n\leq c$.
We have
\begin{align*}
\mathop{\mathrm{CT}}\limits_{\lambda}\frac{\lambda^{-n}}{(1-\lambda^{a}z_1)(1-\lambda^{b}z_2)\underline{(1-\lambda^{c}z_3)}}
&=\mathop{\mathrm{CT}}\limits_{\lambda}\frac{\lambda^{-n}(\lambda^cz_3)^s}{(1-\lambda^{a}z_1)
(1-\lambda^{b}z_2)\underline{(1-\lambda^{c}z_3)}}
\\&=\mathop{\mathrm{CT}}\limits_{\lambda}\frac{-\lambda^{sc-n}z_3^s}{\underline{(1-\lambda^{a}z_1)
(1-\lambda^{b}z_2)}(1-\lambda^{c}z_3)},
\end{align*}
where the last equation follows from Corollary \ref{e-CTECoro}.

Hence, we derive
\begin{align}\label{FirstSecond}
d(n;a,b,c)&=\lim_{z_i\rightarrow 1}\left(\mathop{\mathrm{CT}}\limits_{\lambda}\frac{\lambda^{-n}-\lambda^{sc-n}z_3^s}{\underline{(1-\lambda^{a}z_1)}
(1-\lambda^{b}z_2)(1-\lambda^{c}z_3)}\nonumber
+\mathop{\mathrm{CT}}\limits_{\lambda}\frac{\lambda^{-n}-\lambda^{sc-n}z_3^s}{(1-\lambda^{a}z_1)
\underline{(1-\lambda^{b}z_2)}(1-\lambda^{c}z_3)}\right)
\\&=\lim_{z_i\rightarrow 1}(T(a)+T(b)).
\end{align}

For the reduction from $(a,b)$ to $(a/g_3, b/g_3)$, it is sufficient to perform the extended gcd computation once.
\begin{equation}\label{e-uv}
\gcd(a,b)=g_3 \Rightarrow   au+bv=g_3 \Leftrightarrow u a/g_3+ v b/g_3 =1,  \text{where } u,v\in \mathbb{Z}\setminus \{0\}.
\end{equation} 
In the subsequent subsections, we assume that $a:=\frac{a}{g_3}$ and $b:=\frac{b}{g_3}$, or simply $\gcd(a,b)=1$.

\subsection{Reduction to the Case $b=1$}
If $\gcd(a,b)=1$, then we can reduce a constant term of type $(\ci 1; a,b,c)$ to a constant term of type
$(\ci 1; a,1,c')$. This handles the constant term $T(a)$ in \eqref{FirstSecond}. The constant term $T(b)$ is similar.

Consider a constant term of the following form:
\begin{align*}
\mathop{\mathrm{CT}}\limits_{\lambda}\frac{\lambda^{r_1}\mathbf{z^{m_1}}
-\lambda^{r_2}\mathbf{z^{m_2}}}{\underline{(1-\lambda^{a}z_1)}(1-\lambda^{b}z_2)(1-\lambda^{c}z_3)},
\end{align*}
where $r_1, r_2$ are integers and $\mathbf{z^{m_i}}=z_1^{m_{i1}}z_2^{m_{i2}}z_3^{m_{i3}}$, $m_{ij}\in \mathbb{Q}$. Applying \eqref{e-uv}, we find $u,v$ such that $ua+vb=1$. 
It follows that $\gcd(a,v)=1$, so that Theorem \ref{keyTR} applies to yield
\begin{align}
&\mathop{\mathrm{CT}}\limits_{\lambda}\frac{\lambda^{r_1}\mathbf{z^{m_1}}
-\lambda^{r_2}\mathbf{z^{m_2}}}{\underline{(1-\lambda^{a}z_1)}(1-\lambda^{b}z_2)(1-\lambda^{c}z_3)}
=\mathop{\mathrm{CT}}\limits_{\lambda}
\frac{\lambda^{r_1v}\mathbf{z^{m_1}}-\lambda^{r_2v}\mathbf{z^{m_2}}}{\underline{(1-\lambda^{a}z_1^{\frac{1}{v}})}
(1-\lambda^{bv}z_2)(1-\lambda^{cv}z_3)}\nonumber
\\=&\mathop{\mathrm{CT}}\limits_{\lambda}\frac{\lambda^{r_1v}\mathbf{z^{m_1}}
-\lambda^{r_2v}\mathbf{z^{m_2}}}{\underline{(1-\lambda^{a}z_1^{\frac{1}{v}})}
(1-\lambda^{1-au}z_2)(1-\lambda^{cv}z_3)}
=\mathop{\mathrm{CT}}\limits_{\lambda}\frac{\lambda^{r_1v}\mathbf{z^{m_1}}
-\lambda^{r_2v}\mathbf{z^{m_2}}}{\underline{(1-\lambda^{a}z_1^{\frac{1}{v}})}
(1-\lambda z_1^{\frac{u}{v}}z_2)(1-\lambda^{cv}z_3)}.\label{ToZZNeed11}
\end{align}

\subsection{A Euclid Style Recursive Operation}\label{RecurOper}

We aim to establish a recursion for $T(a)$ and $T(b)$. For this, we relax to consider the following constant term:
\begin{align}
F(\underline{\lambda^{a}\mathbf{z^{\alpha}}}, \lambda\mathbf{z^{\gamma}},  \lambda^{c}\mathbf{z^{\beta}},
\lambda^{r_1}\mathbf{z^{m_1}},\lambda^{r_2}\mathbf{z^{m_2}})
=\mathop{\mathrm{CT}}\limits_{\lambda}\frac{\lambda^{r_1}\mathbf{z^{m_1}}
-\lambda^{r_2}\mathbf{z^{m_2}}}{\underline{(1-\lambda^{a}\mathbf{z^{\alpha}})}
(1-\lambda\mathbf{z^{\gamma}})(1-\lambda^{c}\mathbf{z^{\beta}})},\label{ToZZNeed22}
\end{align}
where $r_1, r_2\in \mathbb{Z}$ and $a, c\in \mathbb{N}$. 
The base cases are: i) if $a=0$, then the constant term is just $0$;
ii) if $a=1$, then the constant term becomes (according to \eqref{AS})
\begin{align*}
\frac{\mathbf{z}^{\mathbf{m_1}-r_1\alpha}
-\mathbf{z}^{\mathbf{m_2}-r_2\alpha}}{(1-\mathbf{z}^{\gamma-\alpha})(1-\mathbf{z}^{\beta-c\alpha})}.
\end{align*}

We require the following two elementary operations.

The \emph{remainder} $\mathrm{rem}^{\star}(\lambda^m,1-\mathbf{z^h}\lambda^a,\lambda)$ and the \emph{signed remainder} $\mathrm{srem}^{\star}(\lambda^m,1-\mathbf{z^h}\lambda^a,\lambda)$ \cite{Xin15} of $\lambda^m$ when dividing by $1-\mathbf{z^h}\lambda^a$ are defined as follows
(different from the usual remainder):
\begin{align*}
\mathrm{rem}^{\star}(\lambda^m,1-\mathbf{z^h}\lambda^a,\lambda)&=\mathbf{z}^{\mathbf{h}\cdot (-\ell)}\lambda^r,\ \ \text{where}\ m=\ell a+r,\ 0< r\leq a;
\\ \mathrm{srem}^{\star}(\lambda^m,1-\mathbf{z^h}\lambda^a,\lambda)&=\mathbf{z}^{\mathbf{h}\cdot (-\ell)}\lambda^r,\ \ \text{where}\ m=\ell a+r,\ -\frac{a}{2}< r\leq \frac{a}{2}.
\end{align*}

\begin{lem}\label{LongDiven}
We have the following recursion:
\begin{align}\label{FRation-and-F}
&F(\underline{\lambda^{a}\mathbf{z^{\alpha}}}, \lambda\mathbf{z^{\gamma}},  \lambda^{c}\mathbf{z^{\beta}},
\lambda^{r_1}\mathbf{z^{m_1}},\lambda^{r_2}\mathbf{z^{m_2}})\nonumber
\\=&\frac{\mathbf{z}^{\mathbf{m_2}^{\prime}-t_2\gamma}-\mathbf{z}^{\mathbf{m_1}^{\prime}-t_1\gamma}}{(1-\mathbf{z}^{\alpha-a\gamma})
(1-\mathbf{z}^{\beta_2-c_2\gamma})}+F(\underline{\lambda^{c_2}\mathbf{z^{\beta_2}}}, \lambda\mathbf{z^{\gamma}},  \lambda^{a}\mathbf{z^{\alpha}},\lambda^{t_2}\mathbf{z^{m_2}}^{\prime},\lambda^{t_1}\mathbf{z^{m_1}}^{\prime}),
\end{align}
where the new parameters are determined through several elementary operations. 
\end{lem}
Let us delay the proof of this lemma.
The first term in \eqref{FRation-and-F} corresponds to the constant term of type $(\ci 2; a_{11},1,a_{31})$ within the flow chart \eqref{FlowChart}. By $c_2\leq \frac{a}{2}$, the second term in \eqref{FRation-and-F} corresponds to the constant term of type $(\ci 1; a_{12},1,a_{32})$ with index $a_{12}\leq \frac{a_{11}}{2}$ in the flow chart \eqref{FlowChart}. Finally, by iteratively applying Lemma \ref{LongDiven} at most $p(\leq \log a)$ steps, we can derive a summation formula for rational functions only involving $z_1,z_2,z_3$.

\begin{proof}
Let $\lambda^{c_1}\mathbf{z}^{\beta_1}=\mathrm{srem}^{\star}(\lambda^c,1-\mathbf{z^{\alpha}}\lambda^a,\lambda)\cdot \mathbf{z}^{\beta}$.
If $0\leq c_1\leq \frac{a}{2}$, then
\begin{align*}
F(\underline{\lambda^{a}\mathbf{z^{\alpha}}}, \lambda\mathbf{z^{\gamma}},  \lambda^{c}\mathbf{z^{\beta}},
\lambda^{r_1}\mathbf{z^{m_1}},\lambda^{r_2}\mathbf{z^{m_2}})
=\mathop{\mathrm{CT}}\limits_{\lambda}\frac{\lambda^{r_1}\mathbf{z^{m_1}}
-\lambda^{r_2}\mathbf{z^{m_2}}}{\underline{(1-\lambda^{a}\mathbf{z^{\alpha}})}
(1-\lambda\mathbf{z^{\gamma}})(1-\lambda^{c_1}\mathbf{z^{\beta_1}})}
\end{align*}
If $-\frac{a}{2}<c_1<0$, then
\begin{align*}
F(\underline{\lambda^{a}\mathbf{z^{\alpha}}}, \lambda\mathbf{z^{\gamma}},  \lambda^{c}\mathbf{z^{\beta}},
\lambda^{r_1}\mathbf{z^{m_1}},\lambda^{r_2}\mathbf{z^{m_2}})
=\mathop{\mathrm{CT}}\limits_{\lambda}\frac{-(\lambda^{r_1}\mathbf{z^{m_1}}
-\lambda^{r_2}\mathbf{z^{m_2}})(\lambda^{-c_1}\mathbf{z^{-\beta_1}})}{\underline{(1-\lambda^{a}\mathbf{z^{\alpha}})}
(1-\lambda\mathbf{z^{\gamma}})(1-\lambda^{-c_1}\mathbf{z^{-\beta_1}})}
\end{align*}
Let
\begin{align*}
\lambda^{c_2}\mathbf{z^{\beta_2}}=\left\{
\begin{aligned}
& \lambda^{c_1}\mathbf{z^{\beta_1}} &\ \ \ \text{if }\ \ & 0\leq c_1\leq \frac{a}{2}, \\
& \lambda^{-c_1}\mathbf{z^{-\beta_1}} &\ \ \  \text{if }\ \  & -\frac{a}{2}<c_1<0. \\
\end{aligned}
\right.
\end{align*}
Let
\begin{align*}
\lambda^{t_1}\mathbf{z}^{\mathbf{m_1}^{\prime}}=\left\{
\begin{aligned}
& \mathrm{rem}^{\star}(\lambda^{r_1},1-\lambda^a\mathbf{z^{\alpha}},\lambda)\cdot\mathbf{z^{m_1}} &\ \ \ \text{if }\ \ & 0\leq c_1\leq \frac{a}{2}, \\
& \mathrm{rem}^{\star}(\lambda^{r_2-c_1},1-\lambda^a\mathbf{z^{\alpha}},\lambda)\cdot\mathbf{z^{m_2-\beta_1}} &\ \ \  \text{if }\ \  & -\frac{a}{2}<c_1<0, \\
\end{aligned}
\right.
\end{align*}
and
\begin{align*}
\lambda^{t_2}\mathbf{z}^{\mathbf{m_2}^{\prime}}=\left\{
\begin{aligned}
& \mathrm{rem}^{\star}(\lambda^{r_2},1-\lambda^a\mathbf{z^{\alpha}},\lambda)\cdot\mathbf{z^{m_2}} &\ \ \ \text{if }\ \ & 0\leq c_1\leq \frac{a}{2}, \\
& \mathrm{rem}^{\star}(\lambda^{r_1-c_1},1-\lambda^a\mathbf{z^{\alpha}},\lambda)\cdot\mathbf{z^{m_1-\beta_1}} &\ \ \  \text{if }\ \  & -\frac{a}{2}<c_1<0, \\
\end{aligned}
\right.
\end{align*}
Then we have
\begin{align*}
&F(\underline{\lambda^{a}\mathbf{z^{\alpha}}}, \lambda\mathbf{z^{\gamma}},  \lambda^{c}\mathbf{z^{\beta}},
\lambda^{r_1}\mathbf{z^{m_1}},\lambda^{r_2}\mathbf{z^{m_2}})
=\mathop{\mathrm{CT}}\limits_{\lambda}\frac{\lambda^{t_1}\mathbf{z}^{\mathbf{m_1}^{\prime}}
-\lambda^{t_2}\mathbf{z}^{\mathbf{m_2}^{\prime}}}{\underline{(1-\lambda^{a}\mathbf{z^{\alpha}})}
(1-\lambda\mathbf{z^{\gamma}})(1-\lambda^{c_2}\mathbf{z^{\beta_2}})}
\\=&\mathop{\mathrm{CT}}\limits_{\lambda}\frac{\lambda^{t_2}\mathbf{z}^{\mathbf{m_2}^{\prime}}
-\lambda^{t_1}\mathbf{z}^{\mathbf{m_1}^{\prime}}}{(1-\lambda^{a}\mathbf{z^{\alpha}})
\underline{(1-\lambda\mathbf{z^{\gamma}})(1-\lambda^{c_2}\mathbf{z^{\beta_2}})}}\ \ \ \text{(by Corollary \ref{e-CTECoro})}
\\=&\frac{\mathbf{z}^{\mathbf{m_2}^{\prime}-t_2\gamma}-\mathbf{z}^{\mathbf{m_1}^{\prime}-t_1\gamma}}{(1-\mathbf{z}^{\alpha-a\gamma})
(1-\mathbf{z}^{\beta_2-c_2\gamma})}+F(\underline{\lambda^{c_2}\mathbf{z^{\beta_2}}}, \lambda\mathbf{z^{\gamma}},  \lambda^{a}\mathbf{z^{\alpha}},\lambda^{t_2}\mathbf{z^{m_2}}^{\prime},\lambda^{t_1}\mathbf{z^{m_1}}^{\prime}).
\end{align*}
This completes the proof.
\end{proof}

\subsection{Dispelling the Slack Variables}\label{DispellSlackV}

In Subsection \ref{RecurOper}, we obtain a summation of the following form.
\begin{align}\label{FiniteZZ}
\sum_{\text{Finite term}}\frac{\mathbf{z^{m_1}}-\mathbf{z^{m_2}}}{(1-\mathbf{z^{\omega}})(1-\mathbf{z^{\theta}})},
\end{align}
where the number of terms is bounded by $\log a+\log b$.

To eliminate the slack variables, as mentioned in \cite{Xin15}, we need to take $z_i\rightarrow 1$ for the final summation.
We first choose a vector $\mu=(\mu_1,\mu_2,\mu_3)$ and substitute $z_i\mapsto \kappa^{\mu_i}$ for $i=1,2,3$, 
and then $\kappa\mapsto e^s$ \footnote{This is a special constant term of GTodd type handled in \cite{XinZhangGTod}. The natural substitution $\kappa \mapsto 1+s$ also works.}.
Finally, we need to calculate the following constant term.
\begin{align}\label{AZ3}
\sum_{\text{Finite term}}\mathop{\mathrm{CT}}\limits_{s}\frac{e^{s\cdot \langle \mathbf{m_1},\mu\rangle}-e^{s\cdot \langle \mathbf{m_2},\mu\rangle}}{(1-e^{s\cdot\langle \omega,\mu\rangle})(1-e^{s\cdot\langle \theta,\mu\rangle})}.
\end{align}
It is essential to choose $\mu$ such that there are no zero denominators in Equation \eqref{AZ3}. We can use random vectors to get $\mu$ (see \cite{DeLoera}).
For \eqref{AZ3}, we also used the linearity of the $\mathrm{CT}$ operator.

Let $h_1=\langle \omega,\mu\rangle$, $h_2=\langle \theta,\mu\rangle$, $h_3=\langle \mathbf{m_1},\mu\rangle$ and $h_4=\langle \mathbf{m_2},\mu\rangle$. 
It is straightforward to verify that
\begin{align}\label{CT5times}
\mathop{\mathrm{CT}}\limits_{s}\frac{e^{h_3s}-e^{h_4s}}{(1-e^{h_1s})(1-e^{h_2s})}
=\frac{(h_4-h_3)(h_1+h_2-h_3-h_4)}{2h_1h_2}.
\end{align}

\section{The Algorithm}

We first give the algorithm, then do complexity analysis, and finally give an example. 

\begin{algorithm}\label{TheAlgorithmBEG}
\DontPrintSemicolon
\KwInput{$n\in \mathbb{N}$, $a,b,c\in \mathbb{P}$, $a<b<c$, and $\gcd(a,b,c)=1$.}
\KwOutput{$d(n;a,b,c)$.}

Use the extended gcd algorithm to compute $u,v$ such that $au+bv=g_3=\gcd(a,b)$,
and apply \eqref{PrimeToAB} to assume $\gcd(a,b)=1$.

By \eqref{FirstSecond}, reduction to two constant terms $T(a), T(b)$.

For $T(a)$, we first reduce it to the case $b=1$ (Equation \eqref{ToZZNeed11}) by using $(a,v)=1$ and Theorem \ref{keyTR}.
Then, iteratively applying Lemma \ref{LongDiven} at most $\log a$ steps, we obtain a summation formula $R_1$ of simple rational functions.

For $T(b)$, we first reduce it to the case $a=1$ by employing $(b,u)=1$ and Theorem \ref{keyTR}.
Then, iteratively applying Lemma \ref{LongDiven} at most $\log b$ steps, we obtain a summation formula $R_2$ for a simple rational function.

We obtain \eqref{FiniteZZ} by $R_1+R_2$ of at most $\log a+\log b$ simple rational functions.

Apply \eqref{CT5times} to each term of \eqref{AZ3}. We obtain $d(n;a,b,c)$.

\caption{Computing $d(n;a,b,c)$.}
\end{algorithm}

\begin{thm}
Let $n\in\mathbb{N}$ and $a,b,c\in \mathbb{P}$. Suppose $a<b<c$ and $\gcd(a,b,c)=1$. 
Then Algorithm \ref{TheAlgorithmBEG} correctly computes $d(n;a,b,c)$ with a computational complexity of $O(\log b)$.
\end{thm}
\begin{proof}
In step 1, the primary task is the computation of $u,v$, which costs $O(\log a)$. Subsequently, we arrive at the $\gcd(a,b)=1$ case.

Step 2 takes several ring operations to reach the constant term $T(a)+T(b)$.

Step 3 performs at most $\log a$ iterative application of Lemma \ref{LongDiven}, each reducing the index of the constant term by half. Each iteration step takes $O(1)$ ring operations since the two operations $\mathrm{rem}(\lambda^m,1-z\lambda^a,\lambda)$ and $\mathrm{srem}(\lambda^m,1-z\lambda^a,\lambda)$ are both $O(1)$. This results in a total cost of $O(\log a)$ for this step.

Step 4 is similar to Step 3. It takes $O(\log b)$ ring operations. 

Now, we arrive at the summation in \eqref{FiniteZZ}.

Step 5 takes $O(\log a+\log b)$, as each constant term is computed by \eqref{CT5times} using $O(1)$ ring operations. 

In total, the algorithm cost $O(\log b)$ ring operations.
\end{proof}

\begin{exa}
Let $a=3$, $b=7$, $c=11$ and $n=25$. We have
\begin{small}
\begin{align*}
d(25;3,7,11)
=\lim_{z_i\rightarrow 1}\left(\mathop{\mathrm{CT}}\limits_{\lambda}\frac{\lambda^{-25}-\lambda^{8}z_3^3}{\underline{(1-\lambda^{3}z_1)}
(1-\lambda^{7}z_2)(1-\lambda^{11}z_3)}
+\mathop{\mathrm{CT}}\limits_{\lambda}\frac{\lambda^{-25}-\lambda^{8}z_3^3}{(1-\lambda^{3}z_1)
\underline{(1-\lambda^{7}z_2)}(1-\lambda^{11}z_3)}\right).
\end{align*}
\end{small}
For the first term, we obtain
\begin{small}
\begin{align*}
&\mathop{\mathrm{CT}}\limits_{\lambda}\frac{\lambda^{2}z_1^9-\lambda^{2}z_1^{-2}z_3^3}{\underline{(1-\lambda^{3}z_1)}
(1-\lambda^{7}z_2)(1-\lambda^{11}z_3)}
=\mathop{\mathrm{CT}}\limits_{\lambda}\frac{\lambda^{2}z_1^9-\lambda^{2}z_1^{-2}z_3^3}{\underline{(1-\lambda^{3}z_1)}
(1-\lambda z_1^{-2}z_2)(1-\lambda^{11}z_3)}
\\=&\mathop{\mathrm{CT}}\limits_{\lambda}\frac{\lambda^{2}z_1^9-\lambda^{2}z_1^{-2}z_3^3}{\underline{(1-\lambda^{3}z_1)}
(1-\lambda z_1^{-2}z_2)(1-\lambda^{-1}z_1^{-4}z_3)}
=\mathop{\mathrm{CT}}\limits_{\lambda}\frac{\lambda^{3}z_1^2z_3^2-\lambda^{3}z_1^{13}z_3^{-1}}{\underline{(1-\lambda^{3}z_1)}
(1-\lambda z_1^{-2}z_2)(1-\lambda z_1^{4}z_3^{-1})}
\\=&\mathop{\mathrm{CT}}\limits_{\lambda}\frac{\lambda^{3}z_1^{13}z_3^{-1}-\lambda^{3}z_1^2z_3^2}{(1-\lambda^{3}z_1)
\underline{(1-\lambda z_1^{-2}z_2)}(1-\lambda z_1^{4}z_3^{-1})}
+\mathop{\mathrm{CT}}\limits_{\lambda}\frac{\lambda^{3}z_1^{13}z_3^{-1}-\lambda^{3}z_1^2z_3^2}{(1-\lambda^{3}z_1)
(1-\lambda z_1^{-2}z_2)\underline{(1-\lambda z_1^{4}z_3^{-1})}}
\\=&\frac{z_1^{19}z_2^{-3}z_3^{-1}-z_1^{8}z_2^{-3}z_3^2}{(1-z_1^7z_2^{-3})(1-z_1^{6}z_2^{-1}z_3^{-1})}
+\frac{z_1z_3^2-z_1^{-10}z_3^5}{(1-z_1^{-11}z_3^3)(1-z_1^{-6}z_2z_3)}.
\end{align*}
\end{small}
Similarly, for the second term, we get
\begin{small}
\begin{align*}
&\mathop{\mathrm{CT}}\limits_{\lambda}\frac{\lambda^{3}z_2^4-\lambda z_2^{-1}z_3^3}{(1-\lambda^{3}z_1)
\underline{(1-\lambda^{7}z_2)}(1-\lambda^{11}z_3)}
=\mathop{\mathrm{CT}}\limits_{\lambda}\frac{\lambda^{3\cdot (-2)}z_2^4-\lambda^{-2} z_2^{-1}z_3^3}{(1-\lambda^{3\cdot (-2)}z_1)
\underline{(1-\lambda^{7}z_2^{-\frac{1}{2}})}(1-\lambda^{11\cdot (-2)}z_3)}
\\=&\mathop{\mathrm{CT}}\limits_{\lambda}\frac{\lambda z_2^{\frac{7}{2}}-\lambda^{5} z_2^{-\frac{3}{2}}z_3^3}{(1-\lambda z_1z_2^{-\frac{1}{2}})
\underline{(1-\lambda^{7}z_2^{-\frac{1}{2}})}(1-\lambda^{-1}z_2^{-\frac{3}{2}}z_3)}
=\mathop{\mathrm{CT}}\limits_{\lambda}\frac{\lambda^6 z_3^{2}-\lambda^{2} z_2^{5}z_3^{-1}}{(1-\lambda z_1z_2^{-\frac{1}{2}})
\underline{(1-\lambda^{7}z_2^{-\frac{1}{2}})}(1-\lambda z_2^{\frac{3}{2}}z_3^{-1})}
\\=&\mathop{\mathrm{CT}}\limits_{\lambda}\frac{\lambda^2 z_2^{5}z_3^{-1}-\lambda^{6} z_3^2}{\underline{(1-\lambda z_1z_2^{-\frac{1}{2}})}(1-\lambda^{7}z_2^{-\frac{1}{2}})(1-\lambda z_2^{\frac{3}{2}}z_3^{-1})}
+\mathop{\mathrm{CT}}\limits_{\lambda}\frac{\lambda^2 z_2^{5}z_3^{-1}-\lambda^{6} z_3^2}{(1-\lambda z_1z_2^{-\frac{1}{2}})
(1-\lambda^{7}z_2^{-\frac{1}{2}})\underline{(1-\lambda z_2^{\frac{3}{2}}z_3^{-1})}}
\\=&\frac{z_1^{-2}z_2^6z_3^{-1}-z_1^{-6}z_2^3z_3^2}{(1-z_1^{-7}z_2^3)(1-z_1^{-1}z_2^{2}z_3^{-1})}
+\frac{z_2^2z_3-z_2^{-9}z_3^8}{(1-z_1z_2^{-2}z_3)(1-z_2^{-11}z_3^7)}.
\end{align*}
\end{small}
We take $\mu=(0,1,1)$. Then we have
\begin{small}
\begin{align*}
&d(25;3,7,11)
\\=& \mathop{\mathrm{CT}}\limits_{s}\frac{e^{-4s}-e^{-s}}{(1-e^{-3s})(1-e^{-2s})}
+\mathop{\mathrm{CT}}\limits_{s}\frac{e^{2s}-e^{5s}}{(1-e^{3s})(1-e^{2s})}
+\mathop{\mathrm{CT}}\limits_{s}\frac{e^{5s}-e^{5s}}{(1-e^{3s})(1-e^{s})}
+\mathop{\mathrm{CT}}\limits_{s}\frac{e^{3s}-e^{-s}}{(1-e^{-s})(1-e^{-4s})}
\\=& 0-\frac{1}{2}+0+\frac{7}{2}=3.
\end{align*}
\end{small}
\end{exa}

\section{Concluding Remark}

The constant term of type $(\ci s; a_1,\dots, a_k)$ corresponds to a denumerant simplicial cone, as detailed in \cite{XinZhang}.
A basic problem is to write such a constant term as the sum of a $N$ rational functions, so that $N$ is small in some sense.
Barvinok's algorithm in computational geometry applies to show that $N$ can be bounded by a polynomial in $\log a_s$
when $k$ is fixed. In algebraic combinatorics, it is conjectured in \cite{XinZhang} that $f^h(\ci s; a_1,\dots, a_k)$ is bounded by a polynomial in $\log a_s$ for a suitable strategy $h$ on the choices of the valid multiplier. Here, $f^h(\ci s; a_1,\dots, a_k)$ denotes the number of terms obtained by repeated application of Theorem \ref{keyTR} and Corollary \ref{e-CTECoro} according to $h$.

A partial result is presented in \cite{XinZhang} for the case where $a_1 \leq 13$, where the optimal strategy $h$ is specified, and the value of $f^h(\ci 1; a_1, a_2, 1, \dots, 1)$ is explicitly determined to be bounded by a polynomial in $k$.

Lemma \ref{LongDiven} is equivalent to giving a simple strategy $h$ so that $f^h(\ci s; a_1,a_2,a_3)\leq \log a_s$:
First transform to $(\ci1;a_1,a_2,1)$ and then use the fact $f^1(\ci 1;a_1,a_2,1)\leq \log a_1$, where $h=1$ means to always choose $1$ as the valid multiplier.
The fact naturally extends to general length $k$ as
$f^1(\ci 1;a_1,a_2,1,\dots,1)\leq (k-2) \log a_1$.

Extending the above idea gives the following partial, but stronger, result:
\begin{prop}
If one application of Theorem \ref{keyTR}
results in a constant term of type $(\ci 1; a_1,\dots, a_k) $ such that
for an $\ell$ we have $|a_i-a_{i'}|\leq 13$ for $2\leq i,i'\leq \ell$, and $a_j\leq 13$ for $j>\ell$, 
then there exists a strategy $h$ such that $f^h(\ci 1; a_1,\dots, a_k) \leq p(k) \log a_1 $ for some polynomial $p(k)$.
\end{prop}
\begin{proof}[Sketched proof]
The case $\ell=2$ follows the same idea as in the flow chart, except that
the outputs are constant terms handled by the partial result.

For $\ell>2$, the recursion yields
$$ f^h(\ci 1; a_1,\dots, a_k)=\sum_{i=2}^{\ell} f^h(\ci i; a_1,\dots, a_k) + \sum_{j=\ell+1}^{k} f^h(\ci j; a_1,\dots, a_k).$$
Since $a_j\leq 13$, each $f^h(\ci j; a_1,\dots, a_k)$ is bounded by a polynomial in $k$.
Since $|a_i-a_{i'}|\leq 13$ for $2\leq i,i'\leq \ell$, one more application of Corollary \ref{e-CTECoro} yields
$$ f^h(\ci i; a_1,\dots, a_k) =f^h(\ci i; b_1,\dots, b_{i-1},a_i, b_{i+1},\dots, b_k) = \sum_{j\neq i} f^h(\ci j; b_1,\dots, b_{i-1},a_i, b_{i+1},\dots, b_k),$$
where $b_j\leq 13$ for all $j\neq 1,i$. This is just the $\ell=2$ case of the proposition.
\end{proof}

Finally, it is natural to propose the following open problem:
\begin{prob}
Let $a<b<c$ and $\gcd(a,b,c)=1$. For any $n\in \mathbb{N}$, is there an algorithm to compute $d(n;a,b,c)$ with a runtime complexity of $O(\log a)$?
\end{prob}

%

\noindent
{\small \textbf{Acknowledgements:}
We are grateful to Zihao Zhang for many useful suggestions.
This work is partially supported by the National Natural Science Foundation of China [12071311].}

\end{document}